\documentclass[11pt,reqno]{amsart}

\usepackage{amsmath,amssymb,amsfonts,amsthm,enumitem,mathtools}
\usepackage[colorlinks,linkcolor=magenta,anchorcolor=blue,citecolor=green]{hyperref}

\newcommand\A{\mathrm{A}} \newcommand\Alt{\mathrm{Alt}} \newcommand\Aut{\mathrm{Aut}} \newcommand\C{\mathrm{C}} \newcommand\Cay{\mathrm{Cay}}
\newcommand\Cen{\mathbf{C}} \newcommand\Hol{\mathrm{Hol}} \newcommand\Inn{\mathrm{Inn}} \newcommand\M{\mathrm{M}} \newcommand\Nor{\mathbf{N}}
\newcommand\Out{\mathrm{Out}} \newcommand\PGL{\mathrm{PGL}}  \newcommand\PSL{\mathrm{PSL}} 
\newcommand\Sy{\mathrm{S}} \newcommand\Sym{\mathrm{Sym}} \DeclareMathOperator{\Soc}{Soc}

\newtheorem{theorem}{Theorem}[section]
\newtheorem{lemma}[theorem]{Lemma}
\newtheorem{proposition}[theorem]{Proposition}
\newtheorem{corollary}[theorem]{Corollary}
\theoremstyle{definition}
\newtheorem{remark}[theorem]{Remark}

\usepackage{needspace}

\begin{document}

\title[Cayley graphs of almost simple groups]{Automorphism groups of Cayley graphs on almost simple groups with normal connection sets}

\author[Cao]{Mengyu Cao}
\address{Institute for Mathematical Sciences, Renmin University of China, Beijing 100086, China}
\email{myucao@ruc.edu.cn}

\author[Lv]{Benjian Lv}
\address{Laboratory of Mathematics and Complex Systems (Ministry of Education), School of Mathematical Sciences, Beijing Normal University, Beijing 100875, China}
\email{bjlv@bnu.edu.cn}

\author[Xia]{Binzhou Xia}
\address{School of Mathematics and Statistics\\The University of Melbourne\\Parkville, VIC 3010\\Australia}
\email{binzhoux@unimelb.edu.au}

\begin{abstract}
We determine the full automorphism group of every connected Cayley graph on an almost simple group with a normal connection set. We also characterize exactly when the full automorphism group is generated by right translations, group automorphisms preserving the connection set, and inversion. Our results substantially generalize several results in the literature, including the known determinations of the automorphism groups of derangement graphs, complete transposition graphs and complete alternating group graphs. As a further application, we establish criteria for a finite nonabelian simple group to admit a graphical doubly regular representation and determine exactly which alternating groups do so, correcting claims in the literature.


\end{abstract}

\maketitle

\section{Introduction}\label{sec:intro}

For a group $G$ and an inverse-closed subset $S$ of $G\setminus\{1\}$, the \emph{Cayley graph} $\Cay(G,S)$ on $G$ with \emph{connection set} $S$ is defined to be the graph with vertex set $G$ such that vertices $x$ and $y$ are adjacent if and only if $yx^{-1}\in S$. A subset of $G$ is said to be \emph{normal} if it is invariant under conjugation by every element in $G$.

Cayley graphs with normal connection sets have been studied under various names in the literature, including \emph{group graphs}~\cite{Fisk1983}, \emph{conjugacy (class) graphs}~\cite{Ito1984}, \emph{quasiabelian Cayley graphs}~\cite{Wang1997}, \emph{normal Cayley graphs}\footnote{This terminology differs from the notion of normal Cayley graphs introduced by Xu in~\cite{Xu1998}.}~\cite{Larose1998}, \emph{central Cayley graphs}~\cite{PonomarenkoVasilev2018}, and \emph{dual Cayley graphs}~\cite{Pan2020}.
These graphs are particularly amenable to spectral methods, as their eigenvalues can be expressed explicitly in terms of the irreducible character values of the underlying group~\cite{Zieschang1988}.
This representation-theoretic description has proved especially useful in Erd\H{o}s--Ko--Rado-type problems for groups~\cite{AhmadiMeagher2015,EllisFriedgutPilpel2011,LongPlazaSinXiang2018,MeagherSpiga2011,MeagherSpiga2016}.
It is also well known that every connected Cayley graph with a normal connection set is Hamiltonian~\cite{Wang1997}.
For further background on Cayley graphs with normal connection sets, see~\cite{Zgrablic2002}.

Determining the (full) automorphism group of a Cayley graph is a fundamental problem that is difficult in general.
The automorphism groups of several well-known families of Cayley graphs with normal connection sets have been determined, including derangement graphs~\cite{DengZhang2011}, complete transposition graphs~\cite{Ganesan2015}, complete alternating group graphs~\cite{HuangHuang2017}, and more recently, Cayley graphs on alternating and symmetric groups whose connection sets consist of all $k$-cycles~\cite{HuangYang2026}.
Note that these Cayley graphs are all defined on almost simple groups, where a group $G$ is said to be \emph{almost simple} with \emph{socle} $T$ if $\Inn(T) \leq G \leq \Aut(T)$ for some finite nonabelian simple group $T$.
Since $\Inn(T) \cong T$, we usually identify $\Inn(T)$ with $T$ when there is no confusion.
The automorphism groups of Cayley graphs on almost simple groups $G$ such that $G/T$ is abelian were studied by Tian and Li~\cite{TianLi2025}, under certain technical conditions on the normal connection set $S$.

In this paper, we prove a general theorem (Theorem~\ref{thm:main}) that completely determines the full automorphism group of every Cayley graph on an almost simple group with a normal connection set. The main results of~\cite{DengZhang2011,Ganesan2015,HabinezaMwambene2025,HuangHuang2017,HuangYang2026,TianLi2025} follow as straightforward corollaries.
To state the theorem, we introduce some terminology and notation.

For a group $G$, let $\iota_G$ be the bijection on $G$ and $L_G,R_G,\Inn_G$ be the homomorphisms from $G$ to $\Sym(G)$ such that for each $x,y\in G$,
\[
x^{\iota_G}=x^{-1},\ \ x^{L_G(y)}=y^{-1}x,\ \ x^{R_G(y)}=xy,\ \  x^{\Inn_G(y)}=y^{-1}xy.
\]
When the group $G$ is clear from the context, we will abbreviate $\iota_G$, $L_G$, $R_G$ and $\Inn_G$ as $\iota$, $L$, $R$ and $\Inn$, respectively. For a subset $S$ of $G$, denote
\[
\Aut(G,S)=\{\alpha\in\Aut(G)\mid S^\alpha=S\}.
\]
For an almost simple group $G$ with socle $T$ and a subset $S\subseteq G$, a coset $Tg$ is said to be \emph{$S$-partial} if $\varnothing\neq S\cap Tg\subsetneq Tg$.

For a graph $\Gamma$ and a vertex $v$, write $\Gamma(v)$ for the neighbourhood of $v$, and for $U\subseteq V(\Gamma)$, write $\Gamma[U]$ for the subgraph induced by $U$.
For a subgroup $X\leq\Aut(\Gamma)$, the \emph{quotient graph} $\Gamma/X$ is the simple graph whose vertices are the $X$-orbits on $V(\Gamma)$, with two distinct orbits adjacent if some vertex of one is adjacent in $\Gamma$ to some vertex of the other. For $\Gamma=\Cay(G,S)$ and $N\trianglelefteq G$, we abbreviate $\Gamma/R_G(N)$ to $\Gamma/N$, whose vertices are the $N$-cosets in $G$.

\begin{theorem}\label{thm:main}
Let $G$ be an almost simple group with socle $T$, let $\Gamma=\Cay(G,S)$ be connected with a normal connection set $S$, let $\overline{P}$ be the subgroup of $G/T$ generated by all the $S$-partial $T$-cosets, let $P$ be the full preimage of $\overline{P}$ under the natural homomorphism $G\to G/T$, and let
\[
D=\{g\in G\mid Sg\setminus\{1\}=S\setminus\{g\}\},\ \ H=PD,\ \ K=TD.
\]
Then $D$, $H$ and $K$ are normal subgroups of $G$ with $K\leq H$, and the following statements hold:
\begin{enumerate}[label=\textup{(\alph*)},leftmargin=*]
\item\label{item:main-blocks} The $H$-cosets and the $K$-cosets both form $\Aut(\Gamma)$-invariant partitions of $G$.
\item\label{item:main-intersection} With the wreath products preserving the partitions of $H$-cosets and $K$-cosets respectively,
\[
\Aut(\Gamma)=\bigl(\Aut(\Gamma[H])\wr\Sym(G/H)\bigr)\cap\bigl(\Sym(K)\wr\Aut(\Gamma/K)\bigr).
\]
\item\label{item:main-local}
$\displaystyle\Aut(\Gamma[H])=
\begin{cases}
\bigl(R_H(H)\rtimes\Aut(H,S\cap H)\bigr)\rtimes\langle\iota_H\rangle&\text{if }D=1\\
\Sym(H)&\text{if }D\neq1.
\end{cases}
$
\end{enumerate}
\end{theorem}

\begin{remark}
The set $D$ in Theorem~\ref{thm:main} has a clear graph-theoretic interpretation: the equality
\[
\Gamma(x)\setminus\{y\}=\Gamma(y)\setminus\{x\}
\]
for vertices $x,y\in G$ defines an equivalence relation whose classes are the $D$-cosets in $G$; in particular, $D$ is the class containing $1$. See the proof of Lemma~\ref{lem:dcosets}. It is also shown in the lemma that either $D=1$, or $\overline{P}=1$ and $D=H=K$.
\end{remark}

\begin{remark}
The graphs $\Gamma[H]$ and $\Gamma/K$ in Theorem~\ref{thm:main} are Cayley graphs of $H$ and $G/K$, respectively. In fact,
\[
\Gamma[H]=\Cay(H,S\cap H)\ \text{ and }\ \Gamma/K=\Cay\bigl(G/K,\{Ks\mid s\in S\setminus K\}\bigr).
\]
\end{remark}

\begin{remark}
Theorem~\ref{thm:main} also readily determines the automorphism group of a possibly disconnected graph $\Gamma=\Cay(G,S)$ with $G$ almost simple and $S$ normal: if $S=\varnothing$, then $\Aut(\Gamma)=\Sym(G)$; otherwise, letting $M=\langle S\rangle$, we have
\[
 \Aut(\Gamma)=\Aut(\Cay(M,S))\wr\Sym(G/M),
\]
where $\Aut(\Cay(M,S))$ is determined in Theorem~\ref{thm:main}, as $M$ is almost simple with socle $T$.
\end{remark}

Combining Theorem~\ref{thm:main} with the classification theorem of Liebeck, Praeger and Saxl~\cite{LPS2010}, which relies on the Classification of Finite Simple Groups, we obtain the following characterizations.

\begin{corollary}\label{cor:criterion}
Let $\Gamma=\Cay(G,S)$ be a connected Cayley graph on an almost simple group $G$ with socle $T$ such that the set of $S$-partial $T$-cosets generates $G/T$. Then the following are equivalent:
\begin{enumerate}[label=\textup{(\alph*)},leftmargin=*]
\item\label{item:as-normal} $S$ is normal in $G$;
\item\label{item:as-iota} $\Aut(\Gamma)$ contains $\iota_G$;
\item\label{item:as-IG} $\Aut(\Gamma)\geq\Inn(G)$;
\item\label{item:as-LG} $\Aut(\Gamma)\geq L(G)$;
\item\label{item:as-LRG} $\Aut(\Gamma)\geq L(G)R(G)\langle\iota_G\rangle$;
\item\label{item:as-canonical} $\Aut(\Gamma)=(R(G)\rtimes\Aut(G,S))\rtimes\langle\iota_G\rangle$ or $\Sym(G)$;
\item\label{item:as-minimal} either $\Aut(\Gamma)$ is primitive on $G$ with $(G,\Soc(\Aut(\Gamma)),\Soc(\Aut(\Gamma))_1)$ not in Table~$\ref{tab:exceptions}$, or $L(T)R(T)$ is a minimal normal subgroup of $\Aut(\Gamma)$.
\end{enumerate}
\end{corollary}

\begin{table}[htbp]
\centering
\caption{$(G,X,X_1)$, where $X=\Soc(\Aut(\Gamma))$, adapted from~\cite[Table~2]{LPS2010}}\label{tab:exceptions}
\small
\begin{tabular}{|l|l|l|l|}
\hline
$G$ & $X$ & $X_1$ & Remark\\
\hline
$\A_5$ & $\PSL_2(59)$ & $\C_{59}\rtimes\C_{29}$ & \\
$\A_7$ & $\A_{11}$ & $\M_{11}$ & \\
$\A_7$ & $\A_{12}$ & $\M_{12}$ & \\
$\A_{p^2-2}$ & $\A_{p^2+1}$ & $\PSL_2(p^2).\C_2$ & $p\equiv3\pmod4$ prime\\
\hline
$\Sy_{p-2}$ & $\A_p$ & $\C_p\rtimes\C_{(p-1)/2}$ & $p\geq7$ prime\\
$\Sy_{p-2}$ & $\A_{p+1}$ & $\PSL_2(p)$ & $p\geq7$ prime\\
\hline
$\Sy_5$ & $\A_9$ & $\PSL_2(8).\C_3$ & \\
$\Sy_5$ & $\mathrm{Sp}_4(4)$ & $\PSL_2(16).\C_2$ & $\Aut(\Gamma)=X.\C_2$\\
$\Sy_5$ & $\mathrm{Sp}_6(2)$ & $\mathrm{G}_2(2)$ & \\
$\Sy_5$ & $\Omega_8^+(2)$ & $\Omega_7(2)$ & \\
$\Sy_5$ & $\mathrm{Sp}_8(2)$ & $\mathrm{O}_8^-(2)$ & \\
\hline
$\PSL_2(16).\C_4$ & $\mathrm{Sp}_6(4)$ & $\mathrm{G}_2(4)$ & $\Aut(\Gamma)\geq X.\C_2$\\
$\PSL_2(16).\C_4$ & $\Omega_8^+(4)$ & $\Omega_7(4)$ & $\Aut(\Gamma)\geq X.\C_2$\\
\hline
$\PSL_3(4).\C_2$ & $\M_{23}$ & $\C_{23}\rtimes\C_{11}$ & \\
$\PSL_3(4).\C_2$ & $\M_{24}$ & $\PSL_2(23)$ & \\
\hline
\end{tabular}
\smallskip
\begin{minipage}{\textwidth}
\end{minipage}
\end{table}

Note that $\Aut(G,S)\geq\Inn(G)$ for each normal subset $S$ of $G$.
Also observe that, for a Cayley graph $\Gamma=\Cay(G,S)$, the connection set $S$ is normal in $G$ if and only if $\Aut(\Gamma)\geq L(G)R(G)\langle\iota\rangle$ (see Lemma~\ref{lem:normalset}). Moreover, if $G$ is nonabelian with $S$ normal, then
\begin{equation}\label{eq:canonical}
\Aut(\Gamma)\geq(R(G)\rtimes\Aut(G,S))\rtimes\langle\iota\rangle
\end{equation}
(see Lemma~\ref{lem:canonicallower}).
The next corollary, from which the results in~\cite{DengZhang2011,Ganesan2015,HuangHuang2017,HuangYang2026,TianLi2025} follow more directly than from Theorem~\ref{thm:main}, characterizes the equality in~\eqref{eq:canonical}.

\begin{corollary}\label{cor:canonicalcriterion}
Let $G$ be an almost simple group, and let $\Gamma=\Cay(G,S)$ be connected with a normal connection set $S$. Then in the notation of Theorem~$\ref{thm:main}$, the following are equivalent:
\begin{enumerate}[label=\textup{(\alph*)},leftmargin=*]
\item\label{item:canonical-equality} $\Aut(\Gamma)=(R(G)\rtimes\Aut(G,S))\rtimes\langle\iota\rangle$;
\item\label{item:canonical-generation} $\Gamma$ is not complete and $P=G$;
\item\label{item:canonical-local} $D=1$ and $H=G$.
\end{enumerate}
\end{corollary}

\begin{remark}
The implication
\ref{item:canonical-generation}$\Rightarrow$
\ref{item:canonical-equality} contains the main result of Tian and Li~\cite[Theorem~1.2]{TianLi2025} as a special case. They assume additionally that $G/T$ is abelian and that $G$ is generated by the intersections of $S$ with certain $S$-partial $T$-cosets. The latter assumption implies $P=G$, while their definition of a $C$-subset ensures that $\Gamma$ is not complete. Their proof contains a gap, which is discussed in Section~\ref{sec:concluding}.
\end{remark}

Introduced by Zgrabli\'{c}~\cite{Zgrablic2002} in 2002, a \emph{graphical doubly regular representation} (\emph{GDRR}) of a group $G$ is a Cayley graph $\Gamma=\Cay(G,S)$ such that
\[
\Aut(\Gamma)=L(G)R(G)\langle\iota\rangle.
\]
A graph is called \emph{arc-transitive} if its automorphism group acts transitively on the set of ordered pairs of adjacent vertices. Pan~\cite{Pan2020}, who used the term \emph{DGRSA}, also considered the existence of GDRRs for finite groups, both in general and under the additional requirement of arc-transitivity.
As a further application of Theorem~\ref{thm:main}, we establish criteria in Theorem~\ref{thm:regularsubset} for determining when a finite nonabelian simple group admits a GDRR. Using these criteria, we show that every such group with outer automorphism group of odd order admits a GDRR (Corollary~\ref{thm:oddout}), and correct the false statement in~\cite{Zgrablic2002} that the alternating group $\A_n$ admits no GDRR for $n\geq4$, as well as the false assertion in~\cite[Theorem~3.3]{Pan2020} that every $\A_n$ with $n\geq5$ admits an arc-transitive GDRR.

\begin{proposition}[{Corrections of~\cite[Corollary~5.5]{Zgrablic2002} and~\cite[Theorem~3.3]{Pan2020}}]\label{prop:altgdrr}
For $n\geq4$, the following statements hold:
\begin{enumerate}[label=\textup{(\alph*)},leftmargin=*]
\item\label{item:alt-gdrr} $\A_n$ admits a GDRR if and only if $n\notin\{4,7,8,12\}$.
\item\label{item:alt-arc-gdrr} $\A_n$ admits an arc-transitive GDRR if and only if $n\notin\{4,6,7,8,12\}$.
\end{enumerate}
\end{proposition}

After collecting the preliminary results in Section~\ref{sec:preliminaries}, we prove Theorem~\ref{thm:main} in Section~\ref{sec:automorphisms}. Section~\ref{sec:applications} proves Corollaries~\ref{cor:criterion} and~\ref{cor:canonicalcriterion}, develops criteria for the existence of GDRRs, and applies them to prove Proposition~\ref{prop:altgdrr}. Section~\ref{sec:concluding} then concludes the paper with some comments.

\section{Preliminaries}\label{sec:preliminaries}

We collect preliminary results on Cayley graphs, primitive permutation groups and wreath products.
The first two lemmas apply to arbitrary groups.

\begin{lemma}\label{lem:normalset}
Let $\Gamma=\Cay(G,S)$ be a Cayley graph. Then the following are equivalent:
\begin{enumerate}[label=\textup{(\alph*)},leftmargin=*]
\item\label{item:ns-normal} $S$ is normal in $G$;
\item\label{item:ns-iota} $\Aut(\Gamma)$ contains $\iota$;
\item\label{item:ns-IG} $\Aut(\Gamma)\geq \Inn(G)$;
\item\label{item:ns-LG} $\Aut(\Gamma)\geq L(G)$;
\item\label{item:ns-LGRG} $\Aut(\Gamma)\geq L(G)R(G)\langle\iota\rangle$.
\end{enumerate}
\end{lemma}

\begin{proof}
It is clear that~\ref{item:ns-normal}$\Leftrightarrow$\ref{item:ns-IG},~\ref{item:ns-LGRG}$\Rightarrow$\ref{item:ns-LG}, and~\ref{item:ns-LGRG}$\Rightarrow$\ref{item:ns-iota}. If $\iota\in\Aut(\Gamma)$, then
\[
\Aut(\Gamma)\geq\langle R(G),\iota\rangle=\langle R(G),R(G)^\iota,\iota\rangle=\langle R(G),L(G),\iota\rangle=L(G)R(G)\langle\iota\rangle.
\]
Thus,~\ref{item:ns-iota}$\Leftrightarrow$\ref{item:ns-LGRG}. Since $\Aut(\Gamma)$ always contains $R(G)$, we have
\[
~\ref{item:ns-IG}\ \,\Leftrightarrow\,\ \Aut(\Gamma)\geq L(G)R(G)\ \,\Leftrightarrow\,\ \ref{item:ns-LG}.
\]
Then it suffices to prove~\ref{item:ns-normal}$\Rightarrow$\ref{item:ns-iota}. Suppose that $S$ is normal in $G$. Then for $x,y\in G$,
\[
yx^{-1}\in S\ \,\Leftrightarrow\,\ xy^{-1}\in S\ \,\Leftrightarrow\,\ y^{-1}(xy^{-1})
y\in S^y\ \,\Leftrightarrow\,\ y^{-1}x\in S\ \,\Leftrightarrow\,\ y^\iota(x^\iota)^{-1}\in S.
\]
This shows that $\iota$ is an automorphism of $\Gamma$, as~\ref{item:ns-iota} states. The proof is thus complete.
\end{proof}

\begin{lemma}\label{lem:blocknormal}
Let $\Gamma=\Cay(G,S)$ be a Cayley graph, and let $X$ be a subgroup of $\Aut(\Gamma)$ containing both $R(G)$ and $\Inn(G)$. Then each block of imprimitivity for $X$ that contains $1$ is a normal subgroup of $G$.
\end{lemma}

\begin{proof}
Let $N$ be a block of imprimitivity for $X$ such that $1\in N$. For each $a\in N$, we have $a=1^{R(a)}\in N\cap N^{R_G(a)}$, which implies that $N=N^{R_G(a)}=Na$. For each $b\in G$, we have $1=1^{\Inn(b)}\in N\cap N^{\Inn(b)}$ and so $N=N^{\Inn(b)}$. Thus, $N$ is a normal subgroup of $G$.
\end{proof}

The next lemma is on Cayley graphs of nonabelian groups with normal connection sets.

\begin{lemma}\label{lem:canonicallower}
Let $G$ be a nonabelian group, and let $\Gamma=\Cay(G,S)$ with $S$ normal in $G$. Then
\[
\Aut(\Gamma)\geq(R(G)\rtimes\Aut(G,S))\rtimes\langle\iota\rangle\cong(G\rtimes\Aut(G,S))\rtimes\C_2.
\]
\end{lemma}

\begin{proof}
Let $N=R(G)\rtimes\Aut(G,S)\leq\Aut(\Gamma)$. For $g\in G$ and $\alpha\in\Aut(G,S)$, we have
\[
R(g)^\iota=L(g)=R(g^{-1})\Inn(g)\in R(G)\Aut(G,S)=N\ \text{ and }\ \alpha^\iota=\alpha\in N.
\]
This shows that $\iota$ normalizes $N$. By Lemma~\ref{lem:normalset}, $\iota\in\Aut(\Gamma)$. If $\iota\in N$, then $\iota\in N_1=\Aut(G,S)$, forcing $G$ to be abelian, a contradiction. Hence $\Aut(\Gamma)\geq\langle N,\iota\rangle=N\rtimes\langle\iota\rangle$, as required.
\end{proof}

For a group $G$, the holomorph of $G$ is denoted by $\Hol(G)$, and we have
\[
\Hol(G)=\Nor_{\Sym(G)}(R(G))=R(G)\rtimes\Aut(G).
\]
Notice that $\iota$ centralizes $\Aut(G)$ and that
\[
R(G)^\iota=L(G)\leq L(G)R(G)=R(G)\Inn(G)\leq R(G)\Aut(G).
\]
We see that $\Hol(G)^\iota=\Hol(G)$, and so
\begin{equation}\label{eq:holiota}
\langle\Hol(G),\iota\rangle=\Hol(G)\langle\iota\rangle=R(G)\Aut(G)\langle\iota\rangle.
\end{equation}

\begin{lemma}\label{lem:regularsimple}
Let $T$ be a finite nonabelian simple group, and let $X$ be a regular subgroup of $\langle\Hol(T),\iota\rangle$ such that $X\cong T$. Then $X=L(T)$ or $R(T)$.
\end{lemma}

\begin{proof}
Let $N=L(T)R(T)=L(T)\times R(T)$. Then $\langle\Hol(T),\iota\rangle/N$ is solvable as $\Out(T)$ is solvable. Since $X\cong T$ is nonabelian simple, it follows that $X\leq N$, and so either $X\in\{L(T),R(T)\}$ or $X=\{L(t)R(t^\alpha)\mid t\in T\}$ for some $\alpha\in\Aut(T)$. Suppose for a contradiction that the latter holds.
Since $X$ is regular, it has trivial intersection with the stabilizer $N_1=\Inn(T)=\{L(t)R(t)\mid t\in T\}$. Hence $t\neq t^\alpha$ for any $t\in T\setminus\{1\}$. In other words, $\alpha$ is a fixed-point-free automorphism of $T$. However, by~\cite{Rowley1995}, a finite group admitting such an automorphism is solvable, contrary to $T$ being nonabelian simple. This completes the proof.
\end{proof}

In~\cite[Theorem~1.4]{LPS2010}, Liebeck, Praeger and Saxl classified the primitive permutation groups with a regular almost simple subgroup. The following lemma is an immediate consequence of their result (one may also prove it directly, as in~\cite[Lemma~3.2]{PonomarenkoVasilev2018} for instance).

\begin{lemma}\label{lem:primitivecase}
Let $G$ be an almost simple group, and let $\Gamma=\Cay(G,S)$ be connected with $S$ normal in $G$. If $\Aut(\Gamma)$ is primitive on $G$, then either $\Gamma$ is a complete graph, or $G$ is simple with $\Aut(\Gamma)=(R(G)\rtimes \Aut(G,S))\rtimes\langle\iota\rangle$.
\end{lemma}

The following lemma is a consequence of~\cite[Theorem~1.4]{LPS2010} (a similar conclusion from the same theorem was obtained in~\cite[Lemma~3.2]{PonomarenkoVasilev2018}, but the exceptional possibility $G=\Sy_5$ and $X=\mathrm{Sp}_8(2)$ listed there can be ruled out, for example, by a direct computation in \textsc{Magma}~\cite{Magma}).

\begin{lemma}\label{lem:lpsprimitive}
Let $G$ be an almost simple group with socle $T$, let $X$ be a primitive permutation group on $G$ containing $L_G(G)R_G(G)$. Then the following statements hold:
\begin{enumerate}[label=\textup{(\alph*)},leftmargin=*]
\item\label{item:lps-nonsimple} If $G\neq T$, then $\Alt(G)\leq X$.
\item\label{item:lps-simple} If $G=T$, then either $\Alt(T)\leq X$ or $X\leq\langle\Hol(T),\iota_T\rangle$.
\end{enumerate}
\end{lemma}

We now turn to Cayley graphs of simple groups with normal connection sets.

\begin{lemma}\label{lem:simplecase}
Let $T$ be a nonabelian simple group, and let $\Gamma=\Cay(T,S)$ with a nonempty normal connection set $S$. Then $\Gamma$ is connected, $\Aut(\Gamma)$ is primitive on $T$, and one of the following holds:
\begin{enumerate}[label=\textup{(\alph*)},leftmargin=*]
\item $S\neq T\setminus \{1\}$ and $\Aut(\Gamma)=(R(T)\rtimes\Aut(T,S))\rtimes\langle\iota\rangle\cong(T\rtimes\Aut(T,S))\rtimes\C_2$;
\item $S=T\setminus \{1\}$, $\Gamma$ is a complete graph, and $\Aut(\Gamma)=\Sym(T)$.
\end{enumerate}
\end{lemma}

\begin{proof}
Since $S$ is nonempty and normal, $\langle S\rangle=T$, and so $\Gamma$ is connected. Lemma~\ref{lem:blocknormal} and the simplicity of $T$ imply that $\Aut(\Gamma)$ is primitive. The conclusion now follows from Lemma~\ref{lem:primitivecase}, since $\Gamma$ is complete if and only if $S=T\setminus\{1\}$.
\end{proof}

We write $\Gamma\circ\Sigma$ for the lexicographic product with vertex set $V(\Gamma)\times V(\Sigma)$, in which $(x,u)$ is adjacent to $(y,v)$ if either $x$ is adjacent to $y$ in $\Gamma$, or $x=y$ and $u$ is adjacent to $v$ in $\Sigma$.
The following classical result is a consequence of Sabidussi~\cite{Sabidussi1959}.

\begin{lemma}\label{lem:sabidussi}
Let $\Gamma$ and $\Sigma$ be finite graphs. If both $\Sigma$ and its complement are connected, then
\[
 \Aut(\Gamma\circ\Sigma)=\Aut(\Sigma)\wr\Aut(\Gamma).
\]
\end{lemma}

For a Cayley graph $\Gamma=\Cay(G,S)$ with an $\Aut(\Gamma)$-invariant partition into right cosets of a subgroup $H\leq G$, identify each coset with $H$ by right translation. Since each induced graph is thereby identified with $\Gamma[H]$, this shows that $\Aut(\Gamma)$ is contained in the wreath product $\Aut(\Gamma[H])\wr\Sym(G/H)$ in its imprimitive action on $G$, where $G/H$ denotes the set of right $H$-cosets. The final lemma of this section expresses the automorphism group of a certain Cayley graph as an intersection of two wreath products.

\begin{lemma}\label{lem:gwintersection}
Let $\Gamma=\Cay(G,S)$ be a finite Cayley graph, and let $H$ and $K$ be normal subgroups of $G$ with $K\leq H$ such that the $K$-cosets and the $H$-cosets both form $\Aut(\Gamma)$-invariant partitions. If $S\setminus H$ is a union of $K$-cosets, then
\[
\Aut(\Gamma)=\bigl(\Aut(\Gamma[H])\wr\Sym(G/H)\bigr)\cap\bigl(\Sym(K)\wr\Aut(\Gamma/K)\bigr),
\]
where the wreath products preserve the partitions of $H$-cosets and $K$-cosets, respectively.
\end{lemma}

\begin{proof}
Let $X=\Aut(\Gamma[H])\wr\Sym(G/H)$ and $Y=\Sym(K)\wr\Aut(\Gamma/K)$. Suppose that $S\setminus H$ is a union of $K$-cosets.
Since the $H$- and $K$-coset partitions are both $\Aut(\Gamma)$-invariant, we have $\Aut(\Gamma)\leq X\cap Y$.

Conversely, let $f\in X\cap Y$. Membership in $X$ implies that $f$ permutes the $H$-cosets and preserves adjacency within each of them. Between distinct $H$-cosets, since $S\setminus H$ is a union of $K$-cosets, the bipartite graph between each pair of $K$-cosets is either complete bipartite or edgeless. Thus, for vertices $u$ and $v$ in distinct $H$-cosets, adjacency in $\Gamma$ is equivalent to adjacency of $Ku$ and $Kv$ in $\Gamma/K$. Then membership in $Y$, together with preservation of the $H$-coset partition, implies that $f$ also preserves adjacency between distinct $H$-cosets. Hence $f\in\Aut(\Gamma)$.
\end{proof}

\begin{remark}
The hypothesis that $S\setminus H$ is a union of $K$-cosets is the \emph{generalized wreath} condition for a Cayley graph $\Cay(G,S)$ if $1<K\leq H<G$; see~\cite[Definition~3.2]{ABDKM2018} and~\cite[Definition~4.1]{DSV2016}.
\end{remark}

\section{Automorphism groups}\label{sec:automorphisms}

Throughout this section, let $G$, $T$, $S$, $\Gamma$, $\overline{P}$, $P$, $D$, $H$ and $K$ be as defined in Theorem~\ref{thm:main}.

\subsection{Special cases}

We start with the following observation.

\begin{lemma}\label{lem:rigidity}
Let $g\in G$, and let $C$ be a nonempty proper subset of $Tg$ that is invariant under conjugation by $T$. Then $\{t\in T\mid Ct=C\}=1$.
\end{lemma}

\begin{proof}
Let $X=\{t\in T\mid Ct=C\}$. Clearly, $X$ is a subgroup of $T$. Since $C^u=C$ for each $u\in T$, one has $Ct^u=(Ct)^u=C^u=C$ for each $t\in X$. Hence $X\trianglelefteq T$, and so $X=1$ or $T$. The latter possibility would make the nonempty set $C$ invariant under right multiplication by $T$, forcing $C=Tg$, a contradiction. Therefore, $X=1$.
\end{proof}

We next show that the $T$-cosets, which need not be blocks in general, are forced to be blocks once there exists an $S$-partial $T$-coset outside $T$.

\begin{lemma}\label{lem:tblocks}
Suppose that $\overline{P}\neq1$. Then the $T$-cosets form an $\Aut(\Gamma)$-invariant partition of $G$, and $\Aut(\Gamma)_T$ is primitive on $T$.
\end{lemma}

\begin{proof}
Let $A=\Aut(\Gamma)$. Since $\overline{P}\neq1$, there exists an $S$-partial coset $Tx\in G/T$ that is distinct from $T$. This implies that $\Gamma$ is not a complete graph, and $G>T$. Then by Lemma~\ref{lem:primitivecase}, $A$ is imprimitive. Choose a minimal nontrivial $A$-block $N$ containing $1$. Lemma~\ref{lem:blocknormal} gives
\[
T\leq N\trianglelefteq G\ \text{ and }\ N<G.
\]
The permutation group $X=A_N^N$ on $N$ induced by the stabilizer $A_N$ is primitive by minimality, and it is straightforward to verify that
\[
L_N(N)R_N(N)\leq X.
\]
If $N=T$, then the lemma already holds. Suppose for a contradiction that $N>T$.
Then Lemma~\ref{lem:lpsprimitive}\ref{item:lps-nonsimple} yields $\Alt(N)\leq X$.
If $Tx\subseteq N$, then since $Tx$ is $S$-partial, the induced graph $\Gamma[N]=\Cay(N,S\cap N)$ is neither empty nor complete. This is excluded by the $2$-transitive group $\Alt(N)\leq X\leq\Aut(\Gamma[N])$.
Consequently, $Tx$ lies outside $N$.

Let $B=S\cap Nx$, $r=|N|$ and $\ell=|B|$.
Then $B$ and $Nx\setminus B$ are nonempty and invariant under conjugation by $N$. Since $\Cen_G(T)=1$, each $N$-conjugacy class in $Nx$ has size at least $2$. Hence $2\leq\ell\leq r-2$.
For $y\in Nx$, as $N=Nx^{-1}y$, the set
\[
\Gamma(y)\cap N=Sy\cap N=(S\cap Nx^{-1})y=B^{-1}y
\]
has size $\ell$. Fix such a vertex $y$. Since $A_N^N\geq\Alt(N)$ is transitive on the set of $\ell$-subsets of $N$, for each $\ell$-subset $C\subseteq N$, there exists $a\in A_N$ such that
\[
C=(\Gamma(y)\cap N)^a=\Gamma(y^a)\cap N.
\]
Thus each $\ell$-subset of $N$ occurs as the neighbourhood in $N$ of some vertex of $\Gamma$. Accordingly,
\[
\binom r\ell\leq|V(\Gamma)|=|G|.
\]
However, $|\Out(T)|\leq |T|/30$ by~\cite[Lemma~2.2]{Quick2004}. This gives a contradiction
\[
|G|=|T||G/T|\leq |T||\Out(T)|\leq\frac{|T|^2}{30}\leq\frac{r^2}{30}<\binom r2\leq\binom r\ell,
\]
as $2\leq\ell\leq r-2$. This completes the proof.
\end{proof}

We next determine the automorphism group when the $S$-partial cosets generate $G/T$.

\begin{lemma}\label{lem:local}
Suppose that $T<G=P$. Then $\Aut(\Gamma)=(R(G)\rtimes\Aut(G,S))\rtimes\langle\iota_G\rangle$.
\end{lemma}

\begin{proof}
Let $A=\Aut(\Gamma)$. By Lemma~\ref{lem:tblocks}, the $T$-cosets form an $\Aut(\Gamma)$-invariant partition of $G$. As a consequence, $A_1$ stabilizes $T$.

We first prove that $A_1$ acts faithfully on $T$. Let $a\in A_1$ fix $T$ pointwise, let $Tx$ be an $S$-partial coset distinct from $T$, and let $C=S\cap Tx^{-1}$. Since both $S$ and $Tx^{-1}$ are invariant under conjugation by $T$, so is $C$. Suppose that $Ty$ is a $T$-coset that is fixed pointwise by $a$. For $u\in Txy$, consider $v=u^a$. Since $Ty$ is fixed pointwise by $a$,
\[
Cu=\Gamma(u)\cap Ty=(\Gamma(u)\cap Ty)^a=\Gamma(u)^a\cap Ty=\Gamma(v)\cap Ty=(S\cap Tyv^{-1})v.
\]
It follows that $Cuv^{-1}=S\cap Tyv^{-1}$ is invariant under conjugation by $T$, and hence
\[
Cuv^{-1}=(Cuv^{-1})^t=C(uv^{-1})^t
\]
for all $t\in T$. Since $(uv^{-1})^{t}(uv^{-1})^{-1}$ lies in $T$, Lemma~\ref{lem:rigidity} implies that $(uv^{-1})^{t}=uv^{-1}$ for all $t\in T$. Hence $uv^{-1}\in\Cen_G(T)=1$, and so $u=v$. Now we have shown that $a$ fixes $Txy$ pointwise if it fixes $Ty$ pointwise.
Then since $G=P$ is generated by the $S$-partial $T$-cosets other than $T$, it follows that $a$ fixes $G$ pointwise, whence $a=1$. Thus, $A_1$ acts faithfully on $T$, as required.

Let $X=A_T^T$. It is primitive by Lemma~\ref{lem:tblocks}. Since it contains $L_T(T)R_T(T)$, Lemma~\ref{lem:lpsprimitive}\ref{item:lps-simple} gives either $\Alt(T)\leq X$ or $X\leq\langle\Hol(T),\iota_T\rangle$. The first possibility is excluded by the argument in the second paragraph (with $N=T$) in the proof of Lemma~\ref{lem:tblocks}. Thus
\begin{equation}\label{eq:restriction}
X\leq\langle\Hol(T),\iota_T\rangle.
\end{equation}

Let $f\in A_1$. To prove the lemma, it suffices to prove that $f\in\Aut(G,S)\langle\iota_G\rangle$. Since $f$ fixes $1$, the restriction $f|_T$ lies in the stabilizer $X_1$. Hence~\eqref{eq:restriction} in conjunction with~\eqref{eq:holiota} implies $f|_T\in\Aut(T)\langle\iota_T\rangle$. Thus, after replacing $f$ by $f\iota_G$ if necessary, we may assume
\[
f|_T=\alpha\in\Aut(T).
\]
Let $t\in T$ and $b=f^{-1}R_G(t)fR_G(t^\alpha)^{-1}\in A$. Then $b$ fixes $1$, and for each $s\in T$,
\[
s^b=\bigl(s^{\alpha^{-1}}t\bigr)^{\alpha}(t^\alpha)^{-1}=\bigl(s^{\alpha^{-1}}\bigr)^\alpha t^\alpha(t^\alpha)^{-1}=s.
\]
Since $A_1$ acts faithfully on $T$, it follows that $b=1$, or equivalently, $R_G(t)f=fR_G(t^\alpha)$. The same computation applies with $L_G$ in place of $R_G$, as $L_G(T)\leq A$ by Lemma~\ref{lem:normalset}. Now
\[
R_G(t^x)f=fR_G\bigl((t^x)^\alpha\bigr)\ \text{ and }\ L_G(t)f=fL_G(t^\alpha)
\]
for all $x\in G$ and $t\in T$. Applying this to the two sides of the equality $(xt^x)^f=(tx)^f$ gives
\begin{equation}\label{eq:semilinear}
x^f(t^x)^\alpha=t^\alpha x^f.
\end{equation}
For $g\in G$, let $\beta(g)\in\Aut(T)$ be defined by $t^{\beta(g)}=t^g$ for all $t\in T$. Then $\beta\colon G\to\Aut(T)$ is a faithful homomorphism, as $\Cen_G(T)=1$, and~\eqref{eq:semilinear} means $\beta(x)\alpha=\alpha\beta(x^f)$, or equivalently,
\[
\beta(x^f)=\beta(x)^{\alpha}.
\]
It follows that, for all $x,y\in G$,
\[
\beta\bigl((xy)^f\bigr)=\beta(xy)^{\alpha}=\beta(x)^\alpha\beta(y)^\alpha=\beta(x^f)\beta(y^f)=\beta\bigl(x^fy^f\bigr),
\]
and so $(xy)^f=x^fy^f$ by the faithfulness of $\beta$. This shows that $f\in\Aut(G)$. Finally, we derive from $f\in A_1$ that $S^f=\Gamma(1)^f=\Gamma(1)=S$. Hence $f\in\Aut(G,S)$, completing the proof.
\end{proof}

\subsection{General case}

We now use the analysis of the special cases in the preceding subsection to prove Theorem~\ref{thm:main} in general. The key is the set $D$.

\begin{lemma}\label{lem:dcosets}
The set $D$ forms a normal subgroup of $G$ with $S\cap D=D\setminus\{1\}$ or $\varnothing$, the $D$-cosets form an $\Aut(\Gamma)$-invariant partition of $G$, and the bipartite graph between each pair of distinct $D$-cosets is either complete bipartite or empty. Moreover, either $D=1$, or $\overline{P}=1$ and $D=H=K$.
\end{lemma}

\begin{proof}
Let $A=\Aut(\Gamma)$, and for each $x\in G$ let
\[
D(x)=\{x\}\cup\{y\in G\setminus\{x\}\mid A\text{ contains the transposition }(x,y)\}.
\]
For pairwise distinct $x,y,z\in G$, conjugating the transposition $(x,y)$ by $(y,z)$ gives $(x,z)$. As a consequence, $y\in D(x)$ implies $D(x)=D(y)$. Thus, for all $x,y\in G$,
\begin{equation}\label{eq:equivalence}
D(x)=D(y)\ \Leftrightarrow\ y\in D(x)\ \Leftrightarrow\ \Gamma(x)\setminus\{y\}=\Gamma(y)\setminus\{x\}.
\end{equation}
This defines an equivalence relation on $G$, whose classes are the sets $D(x)$ and give an $A$-invariant partition.
It also follows from~\eqref{eq:equivalence} that
\[
D(x)=\{y\in G\mid\Gamma(x)\setminus\{y\}=\Gamma(y)\setminus\{x\}\}
\]
for all $x\in G$. Then the definition of $D$ gives $D=D(1)$ and $gD=D(g)$ for all $g\in G$, and Lemma~\ref{lem:blocknormal} yields $D\trianglelefteq G$. Hence the classes of the above equivalence relation are the $D$-cosets.
Moreover, vertices in the same class have identical neighbours outside that class, and each class induces a complete or an empty graph. Accordingly, the bipartite graph between two distinct $D$-cosets is either complete bipartite or empty, and $S\cap D=D\setminus\{1\}$ or $\varnothing$.

Finally, suppose that $D\neq1$. Since $D$ is a normal subgroup of the almost simple group $G$, it follows that $T\leq D$. For each $T$-coset $Tg$ that is distinct from $T$, it is contained either in $D\setminus\{1\}$ or in the $D$-coset $Dg$ distinct from $D$. In the first case, $Tg$ is contained in $S$ or disjoint from $S$, according to $S\cap D=D\setminus\{1\}$ or $\varnothing$. In the second case, since $\Gamma(1)=S$, while $1\in D$ is either adjacent to every vertex of $Dg$ or adjacent to none of them, we again have $Tg$ contained in $S$ or disjoint from $S$. Hence there exist no $S$-partial $T$-cosets distinct from $T$, which means that $\overline{P}=1$. Consequently, $P=T$, and $K=TD=D=PD=H$, completing the proof.
\end{proof}

\begin{lemma}\label{lem:khblocks}
Theorem~$\ref{thm:main}$\ref{item:main-blocks} holds, and $S\setminus H$ is a union of $K$-cosets.
\end{lemma}

\begin{proof}
If $D\neq1$, then $K=H=D$ by Lemma~\ref{lem:dcosets}, and the assertion follows from the same lemma. Now assume $D=1$, so that $H=PD=P$ and $K=TD=T$. By the definition of $P$, no $T$-coset outside $P$ is $S$-partial. This implies that $S\setminus H$ is a union of $K$-cosets.

Towards Theorem~\ref{thm:main}\ref{item:main-blocks}, we first prove that the $K$-cosets form an $\Aut(\Gamma)$-invariant partition. This is stated in Lemma~\ref{lem:tblocks} if $\overline{P}\neq1$, as $K=T$. Assume for the rest of this paragraph that $\overline{P}=1$. Then
\begin{equation}\label{eq:lexproof}
\Gamma\cong(\Gamma/T)\circ\Gamma[T].
\end{equation}
If $S\cap T$ is either $T\setminus\{1\}$ or $\varnothing$, then~\eqref{eq:lexproof} shows that all vertices of $T$ are equivalent under the relation defined by~\eqref{eq:equivalence}, which means that $T\subseteq D$, contradicting $D=1$. Therefore,
\[
\varnothing\subsetneq S\cap T\subsetneq T\setminus\{1\}.
\]
Now both $S\cap T$ and $(T\setminus\{1\})\setminus S$ are nonempty normal subsets of the simple group $T$. They both generate $T$, and so both $\Gamma[T]$ and its complement are connected. Then Lemma~\ref{lem:sabidussi} applied to~\eqref{eq:lexproof} shows that the $T$-cosets form an $\Aut(\Gamma)$-invariant partition.

As $H=P$, it remains to prove that the $P$-cosets also form an $\Aut(\Gamma)$-invariant partition. On $G/T$, join two distinct $T$-cosets when the bipartite graph between them in $\Gamma$ is neither complete bipartite nor empty. We obtain a Cayley graph on $G/T$ with connection set consisting of the $S$-partial cosets other than $T$. Since the partition of $T$-cosets is $\Aut(\Gamma)$-invariant, this Cayley graph is also $\Aut(\Gamma)$-invariant. Its connected components are the $\overline{P}$-cosets, whose unions in $G$ are the $P$-cosets. Thus the $P$-cosets form an $\Aut(\Gamma)$-invariant partition.
\end{proof}

\begin{lemma}\label{lem:unifiedlocal}
Theorem~$\ref{thm:main}$\ref{item:main-local} holds.
\end{lemma}

\begin{proof}
If $D\neq1$, then Lemma~\ref{lem:dcosets} gives $K=H=D$ and $S\cap H=H\setminus\{1\}$ or $\varnothing$. In this case, $\Gamma[H]$ is complete or empty, whence
\[
\Aut(\Gamma[H])=\Sym(H).
\]
Assume for the remainder of the proof that $D=1$. Then $K=TD=T$ and $H=PD=P$.

If $\overline{P}=1$, then $H=P=T$, and it is shown in the second paragraph in the proof of Lemma~\ref{lem:khblocks} that $\varnothing\subsetneq S\cap T\subsetneq T\setminus\{1\}$. In this case, Lemma~\ref{lem:simplecase} gives
\begin{equation}\label{eq:Gamma[H]}
\Aut(\Gamma[H])=(R_H(H)\rtimes\Aut(H,S\cap H))\rtimes\langle\iota_H\rangle.
\end{equation}
Next assume that $\overline{P}\neq1$. The $S$-partial $T$-cosets generate $P/T$. Let $N=\langle S\cap P\rangle$. Then $1\neq N\trianglelefteq P$. Since $P$ is almost simple with socle $T$, it follows that $T\leq N$, and so the generation of $P/T$ leads to $N=P$. This means that $\Gamma[P]=\Cay(P,S\cap P)$ is connected, and hence Lemma~\ref{lem:local} applied to this graph again gives~\eqref{eq:Gamma[H]}.
\end{proof}

\begin{proof}[\rm\textbf{Proof of Theorem~\ref{thm:main}}]
Since $S$ is normal in $G$, conjugation in $G/T$ preserves the set of $S$-partial cosets. Hence $\overline{P}\trianglelefteq G/T$, and so $P\trianglelefteq G$. Moreover, Lemma~\ref{lem:dcosets} shows that $D\trianglelefteq G$. Thus, $H=PD$ and $K=TD$ are normal in $G$ with $K\leq H$. Finally, parts~\ref{item:main-blocks} and~\ref{item:main-local} are Lemmas~\ref{lem:khblocks} and~\ref{lem:unifiedlocal}, respectively, and part~\ref{item:main-intersection} follows from Lemma~\ref{lem:gwintersection}, whose hypothesis is supplied by Lemma~\ref{lem:khblocks}.
\end{proof}

\section{Applications}\label{sec:applications}

As applications of Theorem~\ref{thm:main}, we first prove Corollary~\ref{cor:criterion} and Corollary~\ref{cor:canonicalcriterion}, and then study the existence of GDRRs for nonabelian simple groups, with alternating groups as the principal application.

\subsection{Characterizations}\label{subsec:characterizations}

\begin{proof}[\rm\textbf{Proof of Corollary~\ref{cor:criterion}}]
If~\ref{item:as-normal} holds, then in the notation of Theorem~\ref{thm:main}, we derive from $T\leq P$ and $\overline{P}=G/T$ that $P=G$, which implies $H=PD=G$ and hence~\ref{item:as-canonical} by Theorem~\ref{thm:main}.
Moreover, Lemma~\ref{lem:normalset} gives the equivalence of~\ref{item:as-normal}--\ref{item:as-LRG}, while~\ref{item:as-canonical} clearly implies~\ref{item:as-iota}. Thus~\ref{item:as-normal}--\ref{item:as-canonical} are equivalent.

Let $A=\Aut(\Gamma)$ and $M=L(T)R(T)=L(T)\times R(T)$. We first observe that
\[
\Cen_{\Sym(G)}(M)=1.
\]
Indeed, $M_1=\Inn_G(T)$ fixes only $1$ in $G$. Since $R(G)$ normalizes $M$ and is transitive on $G$, it follows that, for each $g\in G$, the stabilizer $M_g$ fixes only $g$. Thus, a permutation of $G$ centralizing $M$ must fix every $g\in G$.

Suppose that~\ref{item:as-canonical} holds. If $A=\Sym(G)$, then $A$ is primitive, and $(G,\Soc(A),\Soc(A)_1)$ does not occur in Table~\ref{tab:exceptions}. If $A=(R(G)\rtimes\Aut(G,S))\rtimes\langle\iota_G\rangle$, then since conjugation by $\iota_G$ interchanges $L(T)$ and $R(T)$, we deduce that $M$ is a minimal normal subgroup of $A$. In either case,~\ref{item:as-minimal} holds.

Conversely, suppose that~\ref{item:as-minimal} holds.
First assume that $M$ is a minimal normal subgroup of $A$. Let $N=\Nor_{\Sym(G)}(M)$, so that $A\leq N$, and let $f$ be an element in the stabilizer of $1$ in $\Cen_N(L(T))$. Then $f$ fixes $T=1^{L(T)}$ pointwise. Since $M$ is normal in $A$, conjugation by $f$ preserves $\{L(T),R(T)\}$. It follows that $f$ normalizes $R(T)$, as it fixes $L(T)$. For each $t\in T$, the permutations $R(t)^f$ and $R(t)$ therefore belong to $R(T)$ and agree on $T$, and hence are equal. This leads to $f\in\Cen_{\Sym(G)}(M)=1$. Since $R(G)$ is a transitive subgroup of $\Cen_N(L(T))$, the Frattini argument for permutation groups then gives $\Cen_N(L(T))=R(G)$. Similarly, $\Cen_N(R(T))=L(G)$.
Note that the minimal normality of $M$ in $A$ forces some $a\in A$ to interchange $L(T)$ and $R(T)$. As a consequence,
\[
L(G)=\Cen_N(R(T))=\Cen_N(L(T))^a=R(G)^a\leq A,
\]
proving~\ref{item:as-LG}.

Now assume that $A$ is primitive and $(G,\Soc(A),\Soc(A)_1)$ does not occur in Table~\ref{tab:exceptions}. Then by~\cite[Theorem~1.4]{LPS2010}, either $A\geq\Alt(G)$, or $G=T$ and $M\leq A^\sigma\leq\Hol(T)\langle\iota_T\rangle$ for some $\sigma\in\Sym(T)$. In the former case, the $2$-transitivity of $A$ on $G$ implies that $\Gamma$ is complete, and so $S=G\setminus\{1\}$ is normal, as in~\ref{item:as-normal}. In the latter case, the regular subgroup $R(T)^\sigma$ is either $R(T)$ or $L(T)$ by Lemma~\ref{lem:regularsimple}. In this case, either $\sigma$ or $\sigma\iota_T$ normalizes $R(T)$, and hence $\sigma\in\Hol(T)\langle\iota_T\rangle$ normalizes $M$. This implies that $M\leq A$, and in particular $L(T)\leq A$, as in~\ref{item:as-LG}.
\end{proof}

\begin{proof}[\rm\textbf{Proof of Corollary~\ref{cor:canonicalcriterion}}]
If~\ref{item:canonical-local} holds, then $D=1$ rules out the possibility for $\Gamma$ to be complete and leads to $G=H=PD=P$. Hence~\ref{item:canonical-local}$\Rightarrow$\ref{item:canonical-generation}. Moreover,~\ref{item:canonical-generation}$\Rightarrow$\ref{item:canonical-equality} by Corollary~\ref{cor:criterion}.
To complete the proof, we suppose~\ref{item:canonical-equality} and deduce~\ref{item:canonical-local}; that is, we suppose
\begin{equation}\label{eq:canonical-equality}
\Aut(\Gamma)=(R_G(G)\rtimes\Aut(G,S))\rtimes\langle\iota_G\rangle
\end{equation}
and prove $H=G$ and $D=1$.

First, suppose for a contradiction that $H<G$. Fix any $t\in T\setminus\{1\}$ and define $\eta\in\Sym(G)$ to act as conjugation by $t$ on $H$ and fix $G\setminus H$ pointwise. Then the restriction of $\eta$ to $H$ lies in $\Aut(H,S\cap H)\leq\Aut(\Gamma[H])$, while its restriction to every other $H$-coset is the identity. Moreover, since $t\in T\leq K$ and $K\trianglelefteq G$, conjugation by $t$ fixes every $K$-coset in $H/K$, whence $\eta$ induces the identity on $G/K$. Thus, $\eta$ lies in both the wreath products in Theorem~\ref{thm:main}\ref{item:main-intersection}, and hence lies in $\Aut(\Gamma)$.
Since $\eta$ fixes $1$, it follows from~\eqref{eq:canonical-equality} that
\[
\eta\in\Aut(G,S)\rtimes\langle\iota_G\rangle.
\]
As the restriction of $\eta$ on $H$ is conjugation by $t$, we have $\eta\notin\Aut(G,S)\iota_G$. Hence $\eta\in\Aut(G,S)$. Fix any $x\in G\setminus H$. Then for each $h\in H$,
\[
h^\eta=\bigl((hx)x^{-1}\bigr)^\eta=(hx)^\eta(x^{-1})^\eta.
\]
However, since $hx$ and $x^{-1}$ are both in $G\setminus H$, they are both fixed by $\eta$. It follows that $h^\eta=(hx)(x^{-1})=h$.
This shows that $\eta$ fixes $H$ pointwise, contradicting $t\neq1$.

Now we obtain $H=G$. Suppose that $D\neq1$. Then Lemma~\ref{lem:dcosets} gives $D=H=G$, which means that $\Gamma$ is complete or empty. Since $\Gamma$ is connected, it must be complete, and so $\Aut(\Gamma)=\Sym(G)$. This contradicts~\eqref{eq:canonical-equality}, completing the proof.
\end{proof}

\subsection{Existence of GDRRs}\label{subsec:gdrr-existence}

Throughout this subsection, for a finite group $T$, let
\begin{equation}\label{eq:Omega}
\Omega_T=\{C\cup C^{-1}: C\text{ is a non-identity conjugacy class of }T\}.
\end{equation}
Since $\Inn(T)$ acts trivially on $\Omega_T$, the group $\Out(T)$ acts on $\Omega_T$.

\begin{theorem}\label{thm:regularsubset}
Let $T$ be a finite nonabelian simple group, let $\Omega_T$ be defined by~\eqref{eq:Omega}, and let $E$ be the kernel of $\Out(T)$ on $\Omega_T$. Then the following statements hold:
\begin{enumerate}[label=\textup{(\alph*)},leftmargin=*]
\item\label{item:gdrr-kernel} $E$ is an elementary abelian $2$-group.
\item\label{item:gdrr-subset} $T$ admits a GDRR if and only if there exists a subset of $\Omega_T$ whose stabilizer in $\Out(T)$ is trivial.
\item\label{item:gdrr-arc} $T$ admits an arc-transitive GDRR if and only if there exists a member of $\Omega_T$ whose stabilizer in $\Out(T)$ is trivial.
\item\label{item:gdrr-abelian} For $T$ such that $\Out(T)$ is abelian, $T$ admits a GDRR if and only if $E=1$.
\end{enumerate}
\end{theorem}

\begin{proof}
By Feit--Seitz~\cite[Theorem~C]{FS1989}, every class-preserving automorphism of $T$ is inner. Let $e\in E$, and write $e=\Inn(T)\alpha$ with $\alpha\in\Aut(T)$. For each conjugacy class $C$ of $T$, either $C^\alpha=C$ or $C^\alpha=C^{-1}$. Consequently, $\alpha^2$ fixes every conjugacy class of $T$ and is therefore inner. Hence $e^2=1$, proving statement~\ref{item:gdrr-kernel}.

Suppose that $\Gamma=\Cay(T,S)$ is a GDRR; that is, $\Aut(\Gamma)=L(T)R(T)\langle\iota\rangle$. Then $S$ is normal in $T$ by Lemma~\ref{lem:normalset}, and $\Gamma$ is neither empty nor complete, which means that $\varnothing\subsetneq S\subsetneq T\setminus\{1\}$. Since $T$ is simple, $S$ generates $T$. Hence Corollary~\ref{cor:criterion} applies, and its part~\ref{item:as-canonical} gives $\Aut(T,S)=\Inn(T)$. Now $S$ is the union of members in a subset of $\Omega_T$, whose stabilizer in $\Out(T)$ is $\Aut(T,S)/\Inn(T)$ and is therefore trivial.

Conversely, let $\Delta\subseteq\Omega_T$ be a subset with trivial stabilizer in $\Out(T)$. Assume first that $\Out(T)\ne1$. Then $\Delta$ is neither empty nor the whole $\Omega_T$, and the union $S$ of the members of $\Delta$ is therefore a nonempty proper inverse-closed normal subset of $T$. As it generates $T$ with $\Aut(T,S)=\Inn(T)$, Corollary~\ref{cor:criterion} shows that $\Cay(T,S)$ is a GDRR. If $\Out(T)=1$, then take for $S$ a conjugacy class of involutions. This is again a nonempty proper inverse-closed normal subset, and $\Aut(T,S)=\Inn(T)$; hence the same argument gives a GDRR.

Now we have shown~\ref{item:gdrr-subset}. If $T$ admits a GDRR, then there exists a subset of $\Omega_T$ with trivial stabilizer in $\Out(T)$, which implies that $E=1$.
Next suppose that $\Out(T)$ is abelian and $E=1$. In an abelian permutation group, the stabilizer of a point is the kernel of the action on its orbit. Choose one point from each nontrivial $\Out(T)$-orbit on $\Omega_T$. Since distinct orbits are invariant, the stabilizer of the chosen set is the intersection of these orbit kernels, which is $E=1$. Statement~\ref{item:gdrr-subset} then shows that $T$ has a GDRR. This proves statement~\ref{item:gdrr-abelian}.

To prove~\ref{item:gdrr-arc}, first suppose that $\Gamma=\Cay(T,S)$ is an arc-transitive GDRR. Then
\[
\Aut(\Gamma)_1=\Inn(T)\times\langle\iota\rangle,
\]
whose orbits on $T\setminus\{1\}$ are precisely the members $C\cup C^{-1}$ of $\Omega_T$. Thus, arc-transitivity implies that $S=C\cup C^{-1}$ for some conjugacy class $C$ of $T$. As shown in the proof of~\ref{item:gdrr-subset}, the stabilizer of $S$ in $\Out(T)$ is trivial.

Conversely, suppose that some $C\cup C^{-1}\in\Omega_T$ has trivial stabilizer in $\Out(T)$, and let $S=C\cup C^{-1}$. The set $S$ is proper in $T\setminus\{1\}$, since a finite nonabelian simple group has elements of at least two distinct prime orders. Applying the converse argument for~\ref{item:gdrr-subset} to the singleton subset $\{C\cup C^{-1}\}$ shows that $\Cay(T,S)$ is a GDRR. Moreover, since $\Inn(T)\times\langle\iota\rangle$ is transitive on $S$, this GDRR is arc-transitive. This proves~\ref{item:gdrr-arc}.
\end{proof}

An immediate consequence of Theorem~\ref{thm:regularsubset} is the following existence result.

\begin{corollary}\label{thm:oddout}
Let $T$ be a finite nonabelian simple group with $|\Out(T)|$ odd. Then $T$ admits a GDRR.
\end{corollary}

\begin{proof}
By Theorem~\ref{thm:regularsubset}\ref{item:gdrr-kernel}, the kernel of $\Out(T)$ on $\Omega_T$ is an elementary abelian $2$-group. Since $|\Out(T)|$ is odd, this action is faithful. By~\cite[Corollary~1]{Gluck1983}, every permutation group of odd order has a subset with trivial stabilizer. Hence Theorem~\ref{thm:regularsubset}\ref{item:gdrr-subset} shows the existence of a GDRR of $T$.
\end{proof}

For outer automorphism groups of order two, Theorem~\ref{thm:regularsubset} gives a particularly convenient necessary and sufficient condition for a finite nonabelian simple group to admit a GDRR.

\begin{corollary}\label{cor:outtwo}
Let $T$ be a finite nonabelian simple group with $|\Out(T)|=2$, and let $\alpha\in\Aut(T)\setminus\Inn(T)$. Then $T$ admits a GDRR if and only if there exists a conjugacy class $C$ of $T$ such that $C^\alpha\notin\{C,C^{-1}\}$.
\end{corollary}

\begin{proof}
By Theorem~\ref{thm:regularsubset}\ref{item:gdrr-abelian}, $T$ admits a GDRR if and only if the action of $\Out(T)$ on $\Omega_T$ is faithful. Since $\Out(T)$ has order two, this is equivalent to its nonidentity element moving some member $C\cup C^{-1}$ of $\Omega_T$, or equivalently, $C^\alpha\notin\{C,C^{-1}\}$.
\end{proof}

We now apply these criteria to alternating groups.

\begin{proof}[\rm\textbf{Proof of Proposition~\ref{prop:altgdrr}}]
Let $T=\A_n$. First assume that $n=6$. Take
\[
S=(1,2,3)^T\cup(1,2,3,4,5)^T.
\]
Then $S$ is an inverse-closed normal subset of $T$ such that $\varnothing\subsetneq S\subsetneq T\setminus\{1\}$.
The stabilizer of $(1,2,3)^T$ in $\Aut(T)$ is the subgroup induced by conjugation in $\Sy_6$ (see~\cite{Atlas1985}, for example).
Moreover, conjugation by each element of $\Sy_6\setminus T$ sends $(1,2,3,4,5)^T$ to the other $T$-class $(1,2,3,5,4)^T$ of $5$-cycles.
Therefore, $\Aut(T,S)=\Inn(T)$, and hence Corollary~\ref{cor:criterion} shows that $\Cay(T,S)$ is a GDRR.

We next apply Theorem~\ref{thm:regularsubset}\ref{item:gdrr-arc} to show the nonexistence of arc-transitive GDRRs for $n=6$. Note that $\Out(T)\cong\C_2^2$, and the nonidentity conjugacy classes of $T$ are
\[
2A,\quad 3A,\quad 3B,\quad 4A,\quad 5A,\quad 5B
\]
in the notation of~\cite{Atlas1985}. Since automorphisms preserve element orders, every orbit of $\Out(T)$ on $\Omega_T$ has size at most $2$, which is smaller than $|\Out(T)|$. Thus no member of $\Omega_T$ has trivial stabilizer in $\Out(T)$, and Theorem~\ref{thm:regularsubset}\ref{item:gdrr-arc} shows that $\A_6$ has no arc-transitive GDRR.

Next assume that $n\notin\{4,6,7,8,12\}$. In this case, $\Out(T)$ has order two, and $\Aut(T)$ is induced by conjugation in $\Sy_n$. Let
\[
x=
\begin{cases}
(1,2,\ldots,n)&\text{if }n\equiv1\pmod{4}\\
(2,3,\ldots,n)&\text{if }n\equiv2\pmod{4}\\
(2,3,4)(5,6,\ldots,n)&\text{if }n\equiv3\pmod{4}\\
(2,3,4)(5,6,7,8,9)(10,11,\ldots,n)&\text{if }n\equiv0\pmod{4}
\end{cases}
\]
be an element of $\Sy_n$, and let $\mathcal O_1,\ldots,\mathcal O_r$ be the cycles in the cycle decomposition of $x$ (viewing fixed points as $1$-cycles). Observe $x\in\A_n$ and that $\mathcal O_1,\ldots,\mathcal O_r$ have pairwise distinct odd lengths.
Let $y\in\Sy_n$ such that, for each cycle $\mathcal O_i=(v_0,v_1,\dots,v_{k-1})$ of length $k$, the permutation $y$ fixes $v_0$ and maps $v_j$ to $v_{k-j}$ for $j\in\{1,\ldots,k-1\}$. For example, when $n\equiv1\pmod{4}$, the permutation $y=(2,n)(3,n-1)\cdots((n+1)/2,(n+3)/2)$. In all cases, $y\in\A_n$ and $x^y=x^{-1}$.
Let $C=x^T$. Then $C=C^{-1}$ because $x^{-1}=x^y$. Choose an odd permutation $z\in\Sy_n$, and let $\alpha$ be the outer automorphism of $T$ induced by $z$. If $C^\alpha=C$, then $x^z=x^t$ for some $t\in T$, and hence $zt^{-1}\in\Cen_{\Sy_n}(x)$. Since the cycle lengths are distinct and odd, $\Cen_{\Sy_n}(x)$ is generated by the disjoint cycles of $x$ and is contained in $T$, a contradiction. Thus $C^\alpha\notin\{C,C^{-1}\}$, and Corollary~\ref{cor:outtwo} asserts that $T$ has a GDRR.
Since $C$ is a single conjugacy class, $L(T)R(T)=R(T)\rtimes\Inn(T)$ is transitive on the arcs of $\Cay(T,C)$. Thus this GDRR is arc-transitive.

Finally, assume that $n\in\{4,7,8,12\}$. Then each element $x$ of $T$ is such that $x^T\cup(x^T)^{-1}$ is a single conjugacy class of $\Sy_n$ (this can be seen either by direct verification or from the splitting criterion for $\A_n$: a split $\Sy_n$-class has pairwise distinct odd cycle lengths, and for $n\in\{4,7,8,12\}$ the permutation reversing all its cycles is odd). Thus, each inverse-closed normal subset $S$ of $T$ satisfies $\Aut(T,S)\geqslant\Sy_n$. Moreover, Lemma~\ref{lem:normalset} implies that a GDRR of $T$ is necessarily $\Cay(T,S)$ for some inverse-closed normal subset $S$ of $T$. Hence we conclude from~\eqref{eq:canonical} that $T$ has no GDRR.
\end{proof}

\section{Concluding remarks}\label{sec:concluding}

We first give the details of the gap in the proof of the main result of~\cite{TianLi2025}. Let $G$ be an almost simple group with socle $T$, and let $A=\Aut(\Cay(G,S))$, where $S\subsetneq G\setminus\{1\}$ is a generating set for $G$ that is normal and inverse-closed (thus $S$ is a $C$-subset in the sense of~\cite{TianLi2025}).
The conclusion of~\cite[Lemma~2.3]{TianLi2025}, which asserts that the $T$-cosets form an $A$-invariant partition, is in fact false. For a counterexample, take $G$ to be any almost simple group with $|G/T|>2$ such that $G$ has a subgroup $H$ of index $2$ (for example, one may take $G=\Aut(\A_6)$ and $H=\Sy_6$). Then $\Cay(G,G\setminus H)$ is the complete bipartite graph $K_{|H|,|H|}$ with parts $H$ and $G\setminus H$. If $t\in T\setminus\{1\}$ and $h\in H\setminus T$, then the transposition $(t,h)$ is an element of $A$ that maps $T$ to a set that intersects $T$ but is not equal to $T$. Hence $T$ is not a block, disproving~\cite[Lemma~2.3]{TianLi2025}.

In the proof of~\cite[Lemma~2.3]{TianLi2025}, it is asserted that $R(G)\trianglelefteq A$, and hence that $R(T)\trianglelefteq A$. This assertion is false: the inversion $\iota_G$ belongs to $A$, and
\[
\iota_G R(G)\iota_G=L(G)\neq R(G)
\]
as $G$ is nonabelian. Consequently, the product $R(T)A_1$ used there has not been shown to be a subgroup, and the ensuing deduction that the $T$-cosets form an $A$-invariant partition is unsupported. This partition is then used in the proofs of~\cite[Proposition~3.3 and Lemma~3.5]{TianLi2025}, on which the proof of their main theorem depends. Nevertheless, the conclusion of that theorem is valid: as observed after Corollary~\ref{cor:canonicalcriterion}, their generation hypothesis implies $P=G$, and Corollary~\ref{cor:canonicalcriterion} therefore supplies a proof. Our result also removes the assumption that $G/T$ is abelian.

Now we turn to the two false assertions in the literature concerning the existence of GDRRs on alternating groups. The first is the assertion in~\cite[Corollary~5.5]{Zgrablic2002} that no alternating group $\A_n$ with $n\geq4$ admits a GDRR. Our Proposition~\ref{prop:altgdrr}\ref{item:alt-gdrr} already shows that this is incorrect. From the brief proof given in~\cite{Zgrablic2002}, it seems that the mistake lies in overlooking the fact that a conjugacy class of $\Sy_n$ may split into two classes in $\A_n$, and either half may be inverse-closed.

Pan later asserted in~\cite[Theorem~3.3]{Pan2020} that every $\A_n$ with $n\geq5$ admits an arc-transitive GDRR. The proposed graph is $\Cay(\A_n,C)$, where $C$ is the conjugacy class of $3$-cycles. However, conjugation by every odd permutation preserves $C$, and hence $\Aut(\A_n,C)$ is strictly larger than $\Inn(\A_n)$. Thus the proposed graph is not a GDRR. More specifically, for $g=(1,2,3)$ and $n\geq5$,
\[
\Cen_{\Sy_n}(g)=\langle g\rangle\times\Sy_{\{4,\ldots,n\}},
\]
whence the index of $\Cen_{\A_n}(g)$ in $\Cen_{\Sy_n}(g)$ is $2$. In the proof of~\cite[Theorem~3.3]{Pan2020}, this index is incorrectly claimed to be $1$ by an erroneous application of~\cite[Lemma~3.2]{Pan2020}. The separate argument given there for $n=6$ has the same problem: the subgroup of $\Aut(\A_6)$ induced by $\Sy_6$ contains an outer automorphism centralizing $g$. Consequently, the construction in~\cite{Pan2020} does not prove the asserted result. Our Proposition~\ref{prop:altgdrr}\ref{item:alt-arc-gdrr} gives the corrected conclusion: the assertion is false precisely for $n\in\{6,7,8,12\}$.

Finally, the criteria developed in Subsection~\ref{subsec:gdrr-existence} can be used to determine precisely which finite nonabelian simple groups other than the alternating groups admit GDRRs and which admit arc-transitive GDRRs. More precisely, these two problems reduce to determining whether the action of $\Out(T)$ on $\Omega_T$ has, respectively, a subset or a member with trivial stabilizer. Carrying this out for all families requires a substantial case-by-case analysis and lies beyond the main focus of this paper.

\bigskip
\noindent\textsc{Acknowledgement.} Cao is supported by the National Natural Science Foundation of China (12301431). Lv is supported by the National Natural Science Foundation of China (12571347 \& 12131011).

\end{document}